\documentclass[letterpaper, 10 pt, conference]{ieeeconf}  
                                                       
\IEEEoverridecommandlockouts                              
\usepackage{amsmath} 
\usepackage[nolist]{acronym}
\usepackage{xcolor}
\usepackage{amsmath}
\usepackage{amssymb}
\usepackage{mathtools}
\usepackage{subcaption}
\usepackage{hyperref}
\hypersetup{colorlinks=true, linkcolor=black}
\usepackage[ruled]{algorithm2e}
\usepackage{mathrsfs}
\definecolor{ccolor}{RGB}{203,96,21}

\newcommand{\R}{\mathbb{R}} 
\newcommand{\K}{\mathbb{K}} 
 
\newcommand{\N}{\mathbb{N}}

\newcommand{\Lie}{\mathcal{L}}

\newcommand{\A}{\mathcal{A}}
\newcommand{\0}{\mathbf{0}}

\newcommand{\psd}{\mathbb{S}}

\newcommand{\bbmu}{\boldsymbol{\mu}}
\newcommand{\bell}{\boldsymbol{\ell}}

\DeclarePairedDelimiter{\norm}{\lVert}{\rVert}

\DeclarePairedDelimiter{\floor}{\lfloor}{\rfloor}

\DeclarePairedDelimiterX{\inp}[2]{\langle}{\rangle}{#1, #2}

\DeclarePairedDelimiter{\Mp}{\mathcal{M}_+(}{)}

\newtheorem{thm}{Theorem}
\newtheorem{lem}[thm]{Lemma}
\newtheorem{prop}[thm]{Proposition}
\newtheorem{prob}[thm]{Problem}
\newtheorem{cor}[thm]{Corollary}
\newtheorem{defn}[thm]{Definition}

\newtheorem{rmk}[thm]{Remark} 

\usepackage{color}

\title{\LARGE \bf Peak Estimation of Rational Systems using Convex Optimization
}

\author{Jared Miller$^1$,  Roy S. Smith$^1$
\thanks{$^1$J. Miller and R. Smith are with the Automatic Control Laboratory (IfA), Department of Information Technology and Electrical Engineering (D-ITET), ETH Z\"{u}rich, Physikstrasse 3, 8092, Z\"{u}rich, Switzerland
(e-mail: \{jarmiller,  rsmith\}@control.ee.ethz.ch).
}
\thanks{J. Miller and R. Smith  were partially supported by the Swiss National Science Foundation under NCCR Automation, grant agreement 51NF40\_180545. }}
\begin{document}

\maketitle
\thispagestyle{empty}
\pagestyle{empty}


\begin{abstract}
\label{sec:abstract}
This paper presents algorithms that upper-bound the peak value of a state function along trajectories of a continuous-time system with rational dynamics. The finite-dimensional but nonconvex  peak estimation problem is cast as a convex infinite-dimensional linear program in occupation measures. This infinite-dimensional program is then truncated into finite-dimensions using the moment-Sum-of-Squares (SOS) hierarchy of semidefinite programs. Prior work on treating rational dynamics using the moment-SOS approach involves clearing dynamics to common denominators or adding lifting variables to handle reciprocal terms under new equality constraints. Our solution method uses a sum-of-rational method based on absolute continuity of measures. The Moment-SOS truncations of our program possess lower computational complexity and (empirically demonstrated) higher accuracy of upper bounds on example systems as compared to prior approaches.




\end{abstract}
\section{Introduction}
\label{sec:introduction}

Peak estimation is the practice of finding extreme values of a state function $p$ along trajectories $x(t)$ of a dynamical system that evolve starting from an initial set $X_0$. Instances of peak estimation (extremizing $p(x(t))$) include finding the maximum speed of an aircraft, height of a rocket, concentration of a chemical, and current along a transmission line. This work focuses on peak estimation in the case of rational continuous-time dynamics for a state $x \in \R^n$ where:
\begin{align}
    \label{eq:rational_dynamics}
    \dot{x}(t) = f(t, x),
    \end{align}
    \noindent where the rational dynamics $f$ can be represented as 
    \begin{align}
    f(t, x) &= f_0(t, x) + \sum_{\ell=1}^L \frac{N_\ell(t, x)}{D_\ell(t, x)}.\label{eq:rational_dynamics_eq}
\end{align}
The expression in \eqref{eq:rational_dynamics} is a sum-of-rational dynamical system in terms of polynomials $f_0$, $N_\ell,$ and $D_\ell$ (with $L$ finite). Applications of peak estimation for rational systems include systems include finding maximal concentrations in  chemical reaction networks with Michaelis-Menten kinetics or yeast glycolysis, velocities in rigid body kinematics (manipulator equation with rational friction models), and occupancies in network queuing models \cite{nvemcova2018towards, klipp2005systems}. Refer to \cite{nvemcova2018towards} for a detailed survey of applications of rational systems, as well as a formulation of algebraic analysis techniques to establish system properties  such as parameter identifiability and controllability. 

The rational-dynamics peak estimation task considered in this work (maximizing a state function $p$ along system trajectories $x(t \mid x_0)$ evolving in a state space $X \in \R^n$ starting from $X_0 \subseteq X$ with a time horizon of $[0, T]$) is described in Problem \ref{prob:peak_traj}.
\begin{prob}
\label{prob:peak_traj}
Find an initial condition $x_0$ and a stopping time $t^*$ to extremize:
\begin{subequations}
\label{eq:peak_traj}
    \begin{align}
    P^* = & \ \sup_{t^*, x_0} \ p(x(t^* \mid x_0)) &\\
   \textrm{subject to} \quad & \dot{x}(t) = f(t, x(t))  \qquad \textrm{from \eqref{eq:rational_dynamics}}  \label{eq:peak_traj_traj}\\
    & t^* \in [0, T], \quad x_0 \in X_0.
    \end{align}
\end{subequations}
\end{prob}
The \ac{ODE} peak estimation problem in \eqref{eq:peak_traj} is an instance of \iac{OCP} with a free terminal time and zero stage (integral) cost. The finite-dimensional problem \eqref{eq:peak_traj} is generically nonconvex in $(t^*, x_0)$, but can be lifted into a pair of primal-dual infinite-dimensional \acp{LP} in occupation measures \cite{lewis1980relaxation}. Computational solution methods for derived measure \acp{LP} include gridding-based discretization \cite{cho2002linear}, random sampling \cite{mohajerin2018infinite}, and the moment-\ac{SOS} hierarchy of \acp{SDP} \cite{lasserre2001global, henrion2008nonlinear, fantuzzi2020bounding}. Peak estimation \acp{LP} have been developed for dynamical systems such as robustly uncertain systems \cite{miller2021uncertain, miller2023robustcounterpart}, stochastic systems (mean and value-at-risk) \cite{cho2002linear, miller2023chancepeak},  time-delay systems \cite{miller2023delay}, and hybrid systems \cite{miller2023hybrid}. 
Other problem domains in which infinite-dimensional \acp{LP} have been used in the analysis and control of dynamical systems include reachable set estimation and backwards-reachable-set maximizing control \cite{Henrion_2014, majumdar2014convex, korda2013inner, kariotoglou2013approximate, schmid2022probabilistic}, maximum positively invariant set estimation \cite{oustry2019inner}, maximum controlled invariant sets \cite{korda2014convex}, global attractors \cite{goluskin2020attractor, schlosser2021converging}, and long-time averages \cite{tobasco2018optimal}. 

All of the previously mentioned applications of \ac{LP} in dynamical systems analysis and control (in the context of continuous state spaces) use a Lipschitz assumption in dynamics in order to prove that there is no relaxation gap between the infinite-dimensional \ac{LP} and the original finite-dimensional nonconvex program. The rational dynamics in \eqref{eq:rational_dynamics} may fail to be globally Lipschitz (within the domain of $[0, T] \times X$), and therefore falls into the theory of nonsmooth dynamical systems \cite{ambrosio2008transport, ambrosio2014continuity, souaiby2023ensemble}. This work utilizes a sum-of-rational representation from \cite{bugarin2016minimizing} in order to cast \eqref{eq:peak_traj} as an \ac{LP} in measures, and uses the theory of nonsmooth Liouville equations from \cite{ambrosio2008transport} to prove equivalence of optima under compactness, trajectory-uniqueness, and positivity assumptions. Prior work in \ac{SOS}-based analysis of rational functions includes clearing to common denominators \cite{parker2021study} (under positivity), and adding new variables to represent the graph of rational functions as equality constraints \cite{magron2018semidefinite, newton2023rational}.







The main contributions of this work are:
\begin{itemize}
    \item A measure \ac{LP} formulation for peak estimation of rational dynamical systems based on the theory of \cite{bugarin2016minimizing} (for sum-of-rational static optimization).
    \item A proof that there is no relaxation gap between the finite-dimensional nonconvex problem \eqref{eq:peak_traj} and the objectives of infinite-dimensional \acp{LP}
    \item A proof of strong duality between the rational-system peak estimation infinite-dimensional \acp{LP}.
    \item Quantification of the computational complexity in terms of sizes of the \ac{PSD} matrices in \acp{SDP}.
    \item Experiments demonstrating the upper-bounding of the true peak on rational dynamical systems.
\end{itemize}

This paper has the following structure: 
Section \ref{sec:preliminaries} reviews preliminaries such as notation and occupation measures. Section \ref{sec:rat_lp} formulates a sum-of-rational-based measure \ac{LP} for the peak estimation problem in \eqref{eq:peak_traj}. Section \ref{sec:rat_sdp} introduces and applies the moment-\ac{SOS} hierarchy to obtain a nonincreasing sequence of upper-bounds to the true peak value $P^*$. Section \ref{sec:examples} performs peak estimation on example rational dynamical systems. Section \ref{sec:conclusion} concludes the paper. Appendix \ref{app:duality} provides a proof of strong duality between the measure and function \acp{LP}. Appendix \ref{app:other_methods} documents \ac{SOS} programs for pre-existing rational-peak-estimation methods \cite{parker2021study, magron2018semidefinite}.
\section{Preliminaries}
\label{sec:preliminaries}

\begin{acronym}[WSOS]
\acro{BSA}{Basic Semialgebraic}





\acro{LP}{Linear Program}
\acroindefinite{LP}{an}{a}

\acro{OCP}{Optimal Control Problem}
\acroindefinite{OCP}{an}{a}

\acro{ODE}{Ordinary Differential Equation}
\acroindefinite{ODE}{an}{a}

\acro{PSD}{Positive Semidefinite}



\acro{SDP}{Semidefinite Program}
\acroindefinite{SDP}{an}{a}

\acro{SOS}{Sum of Squares}
\acroindefinite{SOS}{an}{a}

\acro{WSOS}{Weighted Sum of Squares}

\end{acronym}

\subsection{Notation}
The set of $n$-dimensional indices with sum less than or equal to a value $d$ is $\N^n_{\leq d}$ ($\alpha \in \N_{\leq d}^n$ if $\alpha \in \N^n$ and $\sum_{i=1}^n \alpha_i \leq d$). The set of polynomials with indeterminate $x$ is $\R[x]$, and the subset of polynomials with degree at most $d$ is $\R[x]_{\leq d}$.

The set of continuous (continuous and nonnegative) functions over $S$ is $C(S)$ ($C_+(S)$). The set of signed (nonnegative) Borel measures over a set $S$ is $\mathcal{M}(S)$ ($\mathcal{M}_+(S)$). The measure of a set $A \subseteq S$ w.r.t. $\mu \in \Mp{S}$ is $\mu(A)$.
The sets $C_+(S)$ and $\mathcal{M}_+(S)$ possess a bilinear pairing $\inp{\cdot}{\cdot}$ that acts by Lebesgue integration: $g \in C_+(S), \mu \in \Mp{S}: \inp{g}{\mu} = \int_S g(s) d \mu(s)$. 
This bilinear pairing is an inner product between $C_+(S)$ and $\Mp{S}$ when $S$ is compact (in which $\Mp{S}$ can be canonically identified as the dual of $C_+(S)$), and the pairing can be extended to integration between $C(S)$ (and sets of more general measurable functions) and $\mathcal{M}(S)$. Given two measures $\mu_1 \in \Mp{S_1}, \mu_2 \in \Mp{S_2}$, the product measure $\mu_1 \otimes \mu_2$ is the unique measure satisfying $\forall A_1 \subseteq S_1, \ A_2 \subseteq S_2: \ (\mu_1 \otimes \mu_2)(A_1 \times A_2) = \mu_1(A_1) \mu_2(A_2)$.
Given a set $A \subseteq S$, the $0/1$ indicator function $I_A$ takes on value $I_A(s)=1$ if $s \in A$ and $I_A(s)=0$ if $s \notin A$ (ensuring that $\inp{I_A(S)}{\mu} = \mu(A)$). The mass of a measure $\mu \in \Mp{S}$ is $\mu(S) = \inp{1}{S}$, and $\mu$ is a probability measure if this mass is 1. The Dirac delta supported at a point $s'$ ($\delta_{s=s'}$) is the unique probability measure such that $\forall g \in C(\mathcal{S}): \inp{g}{\delta_{s=s'}} = g(s')$. The adjoint of a linear operator $\mathscr{L}$ is $\mathscr{L}^\dagger$.



\subsection{Occupation Measures}

An \textit{occupation measure} is a nonnegative Borel measure that contains all possible information about the behavior of (a set of) trajectories of a given dynamical system.

For a given initial condition $x_0 \in X_0$, the occupation measure $\mu_{x(\cdot)} \in \Mp{[0, T] \times X}$ of the trajectory $x(t \mid x_0)$ \eqref{eq:peak_traj_traj} up to a stopping time $t^* \in [0, T]$ satisfies $\forall A \in [0, T], \ B \in X$:
\begin{align}
    \label{eq:single_free_occ}
    &\mu_{x(\cdot)}(A \times B \mid t^*) =
    & \int_{[0, t^*]} I_{A \times B}\left((t, x(t \mid x_0)\right) dt. 
\end{align}
The $(t^*, x_0^*)$-occupation measure $\mu_{x(\cdot)}$ in \eqref{eq:single_free_occ} can also be understood in terms of its pairing with arbitrary continuous (measurable) functions:
\begin{align}
    \forall \omega \in C([0, T] \times X)& & \inp{v}{\mu_{x(\cdot)}} & = \textstyle \int_{0}^{t^*} \omega(t, x(t \mid x_0)) dt.
\end{align}

Occupation measures $\mu_{x(\cdot)}$ in \eqref{eq:single_free_occ} may be defined over a distribution of initial conditions $\mu_0 \in \Mp{X_0}$ (with $x_0 \sim \mu_0$):
\begin{align}
    \label{eq:avg_free_occ}
    &\mu(A \times B \mid t^*) =
    & \int_{X_0} \int_{[0, t^*]} I_{A \times B}\left((t, x(t \mid x_0)\right) dt d \mu_0(x_0). \nonumber
\end{align}

The Lie derivative (instantaneous change) of a test function $v \in C^1([0, T] \times X)$ w.r.t. dynamics \eqref{eq:rational_dynamics} is
\begin{equation}
\label{eq:lie}
    \Lie_f v(t, x) = \partial_t v(t,x) + f(t,x) \cdot \nabla_x v(t,x).
\end{equation}

Any trajectory of \eqref{eq:peak_traj_traj} satisfies the conservation law,
\begin{align}
    v(t^*, x(t \mid x_0)) = v(0, x_0) + \int_{0}^{t^*} \Lie_f v(t', x(t' \mid x_0)) dt'. \label{eq:conservation}
\end{align}

The conservation law in \eqref{eq:conservation} is a Liouville equation, and can be written in terms of the initial measure $\mu_0 = \delta_{x=x_0} \in \Mp{X_0}$, terminal measure $\mu_p = \delta_{t=t^*, x=x(t^* \mid x_0)} \in \Mp{[0, T] \times X}$, and occupation measure $\mu = \mu_{x(\cdot)} \in \Mp{[0, T] \times X}$ for all $v$ as \cite{ambrosio2008transport}
\begin{align}
    \inp{v(t, x)}{\mu_p} &= \inp{v(0, x)}{\mu_0} + \inp{\Lie_f v(t, x)}{\mu}. \label{eq:liou_strong} \\ \intertext{The $\forall v$ imposition in equation \eqref{eq:liou_strong} can be written equivalently in shorthand form (with $\Lie_f^\dagger$ as the adjoint linear operator of $\Lie$) as} 
    \mu_p &= \delta_0 \otimes \mu_0 + \Lie_f^\dagger \mu. \label{eq:liou_weak}
\end{align}

Any triple $(\mu_0, \mu_p, \mu)$ that satisfies \eqref{eq:liou_weak} is a \textit{relaxed occupation measure}; the class of relaxed occupation measures may be larger than the set of superpositions (distributions) of occupation measures arising from trajectories.




\section{Rational Linear Program}
\label{sec:rat_lp}

This section will present convex infinite-dimensional \ac{LP} to perform peak estimation of rational systems.

\subsection{Assumptions}


We will begin with the following assumption:
\begin{itemize}
\item[A1:] If a trajectory satisfies $x(t \mid x_0) \not\in X$ for some $x_0 \in X_0$, then $x(t' \mid x_0) \not\in X$ for all $t' \geq t$.    
\end{itemize}
Further assumptions will be added as needed. 
\begin{rmk}
Assumption A1 is a non-return assumption in the style of \cite{miller2023distance}.
\end{rmk}

\subsection{Measure Program}

Problem \ref{prob:peak_meas} introduces \iac{LP} in measures to produce an upper-bound on Problem \ref{prob:peak_traj} \cite{lewis1980relaxation, cho2002linear}:
\begin{prob}
\label{prob:peak_meas}
Find an initial measure $\mu_0$, a relaxed occupation measure $\mu$, and a peak measure $\mu_p$ to supremize
\begin{subequations}
\label{eq:peak_meas}
\begin{align}
p^* = & \ \sup \quad \inp{p}{\mu_p} \label{eq:peak_meas_obj} \\
    \text{subject to:} \quad & \mu_p = \delta_0 \otimes\mu_0 + \Lie_f^\dagger \mu \label{eq:peak_meas_flow}\\
    & \inp{1}{\mu_0} = 1 \label{eq:peak_meas_prob}\\
    & \mu, \mu_p \in \Mp{[0, T] \times X} \label{eq:peak_meas_peak}\\
    & \mu_0 \in \Mp{X_0}. \label{eq:peak_meas_init}
\end{align}
\end{subequations}
\end{prob}

\begin{thm}
\label{eq:no_relaxation_upper_bound}
    Under Assumption A1, Problem \ref{prob:peak_meas} will upper-bound \ref{prob:peak_traj} with $p^* \geq P^*$.
\end{thm}
\begin{proof}
    Let $(t^*, x_0^*) \in [0, T] \times X_0$ be a feasible point of \eqref{eq:peak_traj}, such that $\forall t \in [0, t^*]: \ x(t \mid x_0^*) \in X$. One such feasible point is the tuple $(0, x^*_0)$ for any $x^*_0 \in X_0$. A feasible relaxed occupation measure $(\mu_0, \mu_p, \mu)$ may be constructed from the trajectory $x(t \mid x_0^*)$: with an initial measure $\mu_0 = \delta_{x=x_0^*},$ a peak measure $\mu_p = \delta_{t=t^*, \ x=x(t \mid x_0^*)},$ and an occupation measure of $\mu= \mu_{x(\cdot)}$. This relaxed occupation measure satisfies constraints \eqref{eq:peak_meas_flow}-\eqref{eq:peak_meas_init} and has objective $\inp{p}{\mu_p} = p(x(t^* \mid x_0^*)).$ The upper-bound $p^* \geq P^*$ is proven because every $(t^*, x_0^*)$ has a measure representation.
\end{proof}

Equality of the objectives of Problem \ref{prob:peak_traj} and \ref{prob:peak_meas} will occur under a set of additional assumptions:

\begin{itemize}
    \item[A2:] The set $[0, T] \times X_0 \times X$ is compact.
    \item[A3:] The cost $p(x)$ is continuous.
    \item[A4:] Trajectories of \eqref{eq:peak_traj_traj} starting at $X_0$ in times $[0, T]$ are unique.
\end{itemize}

\begin{thm}
\label{eq:no_relaxation}
    Under assumptions A1-A4, the relation $p^*=P^*$ will hold.
\end{thm}
\begin{proof}
    By Theorem 3.1 of \cite{ambrosio2014continuity}, imposition of assumption A4 ensures that every relaxed occupation measure $(\mu_0, \mu_p, \mu)$ is supported on the graph of (a superposition of) trajectories of \eqref{eq:peak_traj_traj}. Compactness (A2) and (lower semi-) continuity (A3) are necessary to invoke arguments used by Theorem 2.1 of \cite{lewis1980relaxation}, using the non-smooth Theorem 3.1 of \cite{ambrosio2014continuity} rather than a Lipschitz assumption on dynamics.
\end{proof}

\subsection{Absolute Continuity Formulation}

We will use the sum-of-rationals framework of \cite{bugarin2016minimizing} in order to express \eqref{eq:peak_meas} in a form more amenable to numerical computation, using the moment-\ac{SOS} hierarchy of \acp{SDP}. This sum-of-rationals framework uses the notion of absolute continuity of measures.

\begin{defn}
\label{defn:abscont}
    A measure $\nu \in \Mp{S}$ is \textit{absolutely continuous} to $\mu\in \Mp{S}$ ($\nu \ll \mu$) if,
    for every $A \subseteq X$, $\inp{I_A}{\mu} = 0$ implies that $\inp{I_A}{\nu} = 0$.    
\end{defn}
\begin{defn}{Radon-Nikodym derivative}
    For every pair of absolutely continuous measures $\nu \ll \mu$, there exists a nonnegative function $h(s)$ such that $\forall g \in C(S): \ \inp{g(s)}{\nu(s)} = \inp{g(s) h(s)}{\mu(s)}$.
    This function $h$ is also referred to as the \textit{density} of $\nu$ w.r.t. $\mu$, or as the \textit{Radon-Nikodym} derivative $\frac{d \nu}{d \mu}$.
\end{defn}

Given a dynamics $(f_0, {N_\ell}_{\ell=1}^L, {D_\ell}_{\ell=1}^L)$ 
in \eqref{eq:rational_dynamics} and a relaxed occupation measure $(\mu_0, \mu_p, \mu)$ feasible for \eqref{eq:peak_meas}, we can define a set of per-rational measures $\{\nu_\ell\}_{\ell=1}^L$  (with $\forall \ell: \nu_\ell \in \Mp{[0, T] \times X}$). These per-rational measures will be constructed to satisfy the following condition with respect to the rational denominators in \eqref{eq:rational_dynamics_eq}
\begin{align}
    & \forall \omega \in C([0, T] \times X), \ell \in 1..L: \nonumber \\
    & \qquad \inp{\omega(t, x) D_\ell(t, x)}{\nu_\ell(t, x)} = \inp{\omega(t, x)}{\mu(t, x)}.\label{eq:abscont_per}
    \intertext{This condition will be expressed in condensed notation as}
    &\forall \ell: \ D_\ell^\dagger \nu_\ell = \mu.
\end{align}

\begin{rmk}
    Equation \eqref{eq:abscont_per} is inspired by Equation (7) of \cite{bugarin2016minimizing} for sum-of-rational optimization.
\end{rmk}

We now impose the following assumption:
\begin{itemize}    
    \item[A5:] Each function $D_\ell$ is strictly positive over $[0, T] \times X$.
\end{itemize}

\begin{prop}
    The measures $\nu_\ell$ have finite densities $\frac{d \nu_\ell}{d \mu} = 1/D_\ell$ when A5 is in effect. \label{prop:density_pos}
\end{prop}

The Lie derivative in \eqref{eq:peak_meas_flow} can be expanded using \eqref{eq:abscont_per} as ($\forall v \in C^1([0, T] \times X)$ with assumption A5 in place)
\begin{subequations}
\begin{align}
    \inp{\Lie_f v}{\mu} &= \left\langle \partial_t v + \textstyle \left( f_0 + \sum_{\ell=1}^L (N_\ell/D_\ell)\right) \cdot \nabla_x v, \mu \right\rangle \\
    &\inp{\Lie_{f_0} v + \textstyle \sum_{\ell=1}^L (N_\ell/D_\ell) \cdot \nabla_x v}{\mu}  \\ 
    &= \inp{\Lie_{f_0} v}{\mu} + \textstyle \sum_{\ell=1}^L \inp{(N_\ell/D_\ell) \cdot \nabla_x v(t, x)}{\mu}  \\
    &= \inp{\Lie_{f_0} v}{\mu} + \textstyle \sum_{\ell=1}^L \inp{N_\ell \cdot \nabla_x v}{\nu_\ell}.
\end{align}
\end{subequations}

Problem \ref{prob:peak_rat_meas} uses the sum-of-rational framework to pose a measure peak estimation \ac{LP}:
\begin{prob}
    \label{prob:peak_rat_meas}
Find an initial measure $\mu_0$, a relaxed occupation measure $\mu$, a peak measure $\mu_p$, and a set of per-rational measures $\{\nu_\ell\}_{\ell=1}^L$ to supremize
\begin{subequations}
\label{eq:peak_rat_meas}
\begin{align}
p^*_r = & \ \sup_{\mu_0, \mu, \ \mu_p, \{\nu_\ell\}}  \quad  \inp{p}{\mu_p} \label{eq:peak_rat_meas_obj} \\
   \textrm{s.t.} \quad  & \mu_p = \delta_0 \otimes\mu_0 + \Lie_{f_0}^\dagger \mu + \textstyle \sum_{\ell=1}^\ell (N_\ell \cdot \nabla_x)^\dagger \nu_\ell \label{eq:peak_rat_meas_flow}\\
    & \forall \ell: \  \ D_\ell^\dagger \nu_\ell = \mu \label{eq:peak_rat_meas_abscont}\\
    & \inp{1}{\mu_0} = 1 \label{eq:peak_rat_meas_prob}\\
    & \mu, \mu_p, \{\nu_\ell\} \in \Mp{[0, T] \times X} \label{eq:peak_rat_meas_peak}\\
    & \mu_0 \in \Mp{X_0}. \label{eq:peak_rat_meas_init}
\end{align}
\end{subequations}
\end{prob}

\begin{cor}
\label{eq:no_relaxation_rat}
    Under A1-A5, the objective values of Problem \ref{prob:peak_traj} and \ref{prob:peak_rat_meas} will be equal with $P^* = p^*_r$. 
\end{cor}
\begin{proof}
    Theorem \ref{eq:no_relaxation} ensures that $P^*=p^*$ under A1-A4. The finite nature of the densities $1/D_\ell$ from Proposition \ref{prop:density_pos} (proved by Theorem 2.1 of \cite{bugarin2016minimizing}) ensures that the absolute-continuity-based construction process in \eqref{eq:abscont_per} will result in each $\nu_\ell$ having identical support to $\mu$. As a result, $(\mu_0, \mu_p, \mu)$ from \eqref{eq:peak_rat_meas_flow} will form a relaxed occupation measure, thus ensuring no relaxation gap by Theorem 3.1 of \cite{ambrosio2014continuity} (used in Theorem \ref{eq:no_relaxation}).
\end{proof}

\begin{rmk}
    If A5 is not imposed, then it is possible for the measure $\nu_\ell$ from \eqref{eq:abscont_per} to be unconstrained (with $D_\ell = 0$ at some $(t, x)$). It therefore cannot be presumed that the supports of $\nu_\ell$ and $\mu$ are identical. These degrees of freedom in $\nu_\ell$ could allow for the strict (and possibly unbounded) upper-bound $p^*_r > p^*$.
\end{rmk}


\subsection{Function Program}

The dual \ac{LP} for \ref{prob:peak_rat_meas} is contained in Problem \ref{prob:peak_rat_cont}:
\begin{prob}
    \label{prob:peak_rat_cont} Find a $C^1$ auxiliary function $v$, a scalar $\gamma$, and per-rational continuous functions $\{q_\ell\}_{\ell=1}^L$ to infimize
\begin{subequations}
\label{eq:peak_rat_cont}
    \begin{align}
      d^*_r = &\ \inf_{\gamma \in \R, v, \{q_\ell\}} \gamma \\
      \textrm{s.t.} \quad &\forall x \in X_0: \nonumber\\
      & \quad \gamma \geq v(0, x) \\
      &\forall (t, x) \in [0, T] \times X: \nonumber \\
      & \quad v(t, x) \geq p(x) \\
      & \quad -\Lie_{f_0} v(t, x) \textstyle -\sum_{\ell=1}^L q_\ell(t, x) \geq 0 \quad \\
      &\forall (t, x) \in [0, T] \times X, \ \forall \ell\in 1..L: \nonumber \\
      & \quad D_\ell(t, x) q_\ell(t, x) - N_\ell(t, x) \cdot \nabla_x v(t, x) \geq 0  \\
      & v \in C^1([0, T] \times X) \\
      & \forall \ell: \ q_\ell \in C([0, T] \times X).
    \end{align}
\end{subequations}
\end{prob}

\begin{lem}
\label{lem:finite_mass}
    The mass of all feasible measures $(\mu_0, \mu_p, \mu, \{\nu_\ell\})$ for solutions to Problem \ref{prob:peak_rat_meas} are finite under A1-A5.
\end{lem}
\begin{proof}
The mass of $\mu_0$ is constrained to 1 by \eqref{eq:peak_rat_meas_prob}. This pins the mass of $\mu_p$ to 1 by letting $v(t, x) = 1$ be a test function to \eqref{eq:peak_rat_meas_flow}. Applying $v(t, x) = t$ to \eqref{eq:peak_rat_meas_flow} results in $\inp{1}{\mu} = \inp{t}{\mu_p} \leq T$. The masses of $\nu_\ell$ are set to $\inp{1}{\nu_\ell} = \inp{1/D_\ell}{\mu}$. The quantities $\inp{1/D_\ell}{\mu}$ are finite by A5 (because $D_\ell>0$ over $[0, T] \times X$), such that the finite mass bound $\forall \ell: \inp{1}{\nu_\ell} \leq T \sup_{(t, x) \in [0, T] \times X} D_\ell (t, x)$ is respected.
\end{proof}

\begin{thm}
\label{thm:strong_dual}
    Under A1-A5, programs \eqref{eq:peak_rat_meas} and \eqref{eq:peak_rat_cont} will satisfy $p^*_r = d^*_r$ (strong duality).
\end{thm}
\begin{proof}
    See Appendix \ref{app:duality}.
\end{proof}

    
\section{Finite-Dimensional Truncation}
\label{sec:rat_sdp}

Program \eqref{eq:peak_rat_cont} must be discretized into a finite-dimensional program in order to admit tractable numerical solutions by computational means. We will first introduce the moment-\ac{SOS} hierarchy of \acp{SDP}, and then use this  hierarchy in order to perform  finite-dimensional truncations of \eqref{eq:peak_rat_cont}.

\subsection{Sum of Squares Background}

A polynomial $\theta \in \R[x]$ is \ac{SOS} if there exists a tuple of polynomials 
$\{\phi_j\}_{j=1}^{j_{\max}} \in \R[x]^{j_{\max}}$ 
such that $p(x) = \sum_{j=1}^{j_{\max}} \phi_j(x)^2$. 
The set of all \ac{SOS} polynomials in indeterminates $x$ is marked as $\Sigma[x] \subset \R[x]$, and the bounded-degree subset of \ac{SOS} polynomials with degree less than or equal to $2k$ is $\Sigma[x]_{\leq 2k}$. The set of \ac{SOS} polynomials is a strict subset of the cone of nonnegative polynomials, with equality holding only in the cases of univariate polynomials, general quadratics, or bivariate quartics \cite{hilbert1888darstellung, blekherman2006there}. To each \ac{SOS} polynomial $\theta$, there exists a 
(nonunique) tuple of a size $s \in \N$, a polynomial vector $m(x) \in \N^s$, and a \ac{PSD} \textit{Gram} matrix $Q \in \psd_+^s$ such that $\theta(x) = m(x)^T Q m(x)$. When $x$ has dimension $n$ and $m(x)$ is chosen to be the vector of monomials of degrees $0$ to $k$, the size $s$ is $\binom{n+d}{d}$. Testing membership of a polynomial in the \ac{SOS} cone can therefore be done using \acp{SDP} \cite{parrilo2000structured}.

\Iac{BSA} set is a set defined by a finite number of bounded-degree polynomial inequality constraints. For any \ac{BSA} set $\K = \{x \mid g_j(x) \geq 0, j\in 1\ldots N_g\},$ the \ac{WSOS} cone $\Sigma[\K]$ is the class of polynomials $\theta \in \R[x]$ that admit the following representation in terms of \ac{SOS} polynomials $(\sigma_0, \{\sigma_j\})$:
\begin{subequations}
\label{eq:putinar}
    \begin{align}
        & \theta(x) = \sigma_0(x) + \textstyle \sum_j {\sigma_j(x)g_j(x)} \\
        &\exists  \sigma_0(x) \in \Sigma[x], \quad \forall j \in 1\ldots N_g: \sigma_j(x) \in \Sigma[x].\label{eq:putinar_variables}
    \end{align}
\end{subequations}

The set $\K$ is ball-constrained if there exists an $R \geq 0$ such that $R - \norm{x}^2_2 \in \Sigma[\K]$. Every compact set whose bounding radius $R$ is known may be rendered ball-constrained by appending the redundant constraint $R - \norm{x}^2_2$ to the description of $\K$. Every positive polynomial over a ball-constrained set $\K$ is also a member of $\Sigma[\K]$ (Putinar Positivestellensatz \cite{putinar1993compact}).
The multipliers $\sigma_j$ from \eqref{eq:putinar} certifying this positivity may generically have degrees that are exponential in $n$ and degree of $\theta$ \cite{nie2007complexity}. 

The truncated \ac{WSOS} cone $\Sigma[\K]_{\leq 2k}$ is the class of polynomials such that $\deg(\sigma_0) \leq 2k$ and $\forall j: \ \deg(\sigma_j) \leq 2k$. The process of replacing a nonconvex polynomial inequality constraint with \iac{WSOS} constraint and increasing the degree until convergence is called the moment-\ac{SOS} hierarchy.




\subsection{SOS Program}

We will impose the following constraints in order to utilize the moment-\ac{SOS} hierarchy:

\begin{itemize}
    \item[A6:] $X_0$ and $X$ are ball-constrained \ac{BSA} sets.
\end{itemize}

For a fixed degree $k \in \N$ and index $\ell \in 1\ldots L$, let us define the following degree of:
\begin{equation}
    \varepsilon_\ell = \max(\deg(D_\ell), \deg(N_\ell)-1).
\end{equation}

The degree-$k$ \ac{SOS} truncation of \eqref{eq:peak_rat_cont} is:
\begin{prob}
\label{prob:peak_rat_sos}
Find a scalar $\gamma$ and  polynomials $v, \{q_\ell\}_{\ell=1}^L$ to minimize
\begin{subequations}
\label{eq:peak_rat_sos}
    \begin{align}
      d^*_{r, k} = &\ \min_{\gamma \in \R, v, \{q_\ell\}} \gamma \\      
      & \gamma - v(0, x) \in \Sigma[X_0]_{\leq 2k}\\
      & \quad v(t, x) - p(x) \in \Sigma[([0, T] \times X)]_{\leq 2k}\\
      &-\Lie_{f_0} v(t, x) \textstyle -\sum_{\ell=1}^L q_\ell(t, x) \label{eq:peak_rat_sos_big} \\
      & \ \qquad \in \Sigma[([0, T] \times X)]_{\leq 2k + 2 \floor{(\deg(f_0) -1)/2}} \nonumber   \\
      &\forall \ell\in 1\ldots L: \nonumber \\
      & \quad D_\ell(t, x) q_\ell(t, x) - N_\ell(t, x) \cdot \nabla_x v(t, x) \\
      & \ \qquad \in  \Sigma[([0, T] \times X)]_{ \leq 2k + 2\floor{\varepsilon_\ell/2}} \nonumber   \\
      & v \in \R[t, x]_{\leq 2k}\\
      & \forall \ell \in 1\ldots L: \ q_\ell \in \R[t, x]_{\leq 2k}.
    \end{align}
\end{subequations}
\end{prob}

The following lemma ensuring finiteness of measure masses is required to prove convergence of Problem \ref{prob:peak_rat_sos} to the optimal value of Problem \ref{prob:peak_rat_meas} as $k\rightarrow \infty$:
\begin{thm}
    Under assumptions A1-A6, the finite truncations will converge from above as $\lim_{k\rightarrow \infty} d^*_{r, k} = P^*$
\end{thm}
\begin{proof}
    We first note that $P^* = d^*_r$ under A1-A5 by Corollary \ref{eq:no_relaxation_rat} using Theorem \ref{thm:strong_dual} (strong duality) and Lemma \ref{lem:finite_mass} (finite mass).

After noting that the masses of all feasible measures are bounded by Lemma \ref{lem:finite_mass}, convergence in objective to $d^*_r = P^*$ is proven by Corollary 8 of \cite{tacchi2022convergence}.
\end{proof}

\subsection{Computational Complexity}
The dominant-size Gram \ac{PSD} constraint of program \eqref{eq:peak_rat_sos} occurs at \eqref{eq:peak_rat_sos_big}, and has size $\binom{n+1+2k + 2 \floor{\deg(f_0)/2}}{n+1}$. When using an interior point method, the scaling of \eqref{eq:peak_rat_sos} therefore grows as $O(n^{6k})$ (nominally) or exponentially (in a degenerate case from Proposition 6 of \cite{gribling2023note}). Further complexity reductions such as symmetry and term sparsity may be employed if present to reduce computation time, but exploiting sparsity may lead to different finite-degree SOS optimal values.

\section{Numerical Examples}

\label{sec:examples}

Julia code to generate all examples in this work is publicly available online\footnote{\href{http://doi.org/10.5905/ethz-1007-711}{http://doi.org/10.5905/ethz-1007-711}}. \ac{SOS} programs were posed using a Correlative-Term-Sparsity interface (CS-TSSOS) \cite{wang2022cs, wang2023exploiting}. The \ac{SOS} programs were converted to \acp{SDP} using \texttt{JuMP} \cite{Lubin2023jump}. All \acp{SDP} were then solved by Mosek 10.1 \cite{mosek101}.


\subsection{Two-Species Chemical Reaction Network}

The first example involves analysis of a chemical reaction network. The states $(x_1, x_2)$ represent the nonnegative concentrations of the two species. The species undergo degradation at a linear rate, and promotion according to Michaelis-Menten kinetics (therefore possessing a globally asymptotically stable equilibrium point in the nonnegative orthant \cite{blanchini2023michaelis}). The relevant dynamics evolving in $X = [0, 1]^2$ over a time horizon of $T = 6$ are:
\begin{subequations}
\label{eq:ex_mm}
    \begin{align}
        \dot{x}_1 &= -\frac{3}{4} x_1 + \frac{1}{1+4.5 x_2}\\
        \dot{x}_2 &= -\frac{9}{16}x_2 + \frac{1.25}{1+6.75 x_1}.
    \end{align}
\end{subequations}

We note that the unique equilibrium point of \eqref{eq:ex_mm} occurs at $x_{eq} = [0.3203, 0.7027].$
Both of the denominators in \eqref{eq:ex_mm} are positive over $X$, thus satisfying assumption A5. Peak estimation is performed to bound the maximal value of $p(x) = x_2$ for trajectories beginning in the disc initial set of $X_0 = \{x \mid 0.3^2 - (x_1-0.3)^2 - (x_2 - 0.3)^2\}$.

A lower-bound on $p(x) = x_2$ acquired from gridded numerical ODE-sampling is $0.8157$. Table \ref{tab:mm} reports upper-bounds acquired by finite-degree \ac{SOS} truncations in \eqref{eq:peak_rat_sos}, as well as bounds discovered by comparison methods from Appendix \ref{app:other_methods} at the same degrees for $v(t, x)$. Table \ref{tab:mm_time} reports the time taken for Mosek to return solutions in Table \ref{tab:mm}. All methods return an upper bound of $P^* \leq 1$ at degree $k=1$. Our method (Problem \ref{prob:peak_rat_sos}) returns the lowest peak estimate at each degree $k$ for this experiment.

\begin{table}[h]
    \centering
    \caption{Bounds for Michaelis-Menten Network \eqref{eq:ex_mm}}
    \begin{tabular}{r|ccccc}
Degree $k$                  &  2      & 3      & 4      & 5    & 6  \\ \hline
Sum-of-rational \eqref{eq:peak_rat_sos}       & 0.8522 & 0.8159 & 0.8159 & 0.8159 & 0.8159 \\
Lifted \eqref{eq:peak_rat_sos_lift} & 0.9242 & 0.8200 & 0.8189 & 0.8170 & 0.8185 \\
Cleared    \eqref{eq:peak_rat_sos_clear} & 0.9202 & 0.8306 & 0.8237 & 0.8225 & 0.8210
    \end{tabular}    
    \label{tab:mm}
\end{table}

\begin{table}[h]
    \centering
    \caption{Timing (seconds) for Table \ref{tab:mm}}
    \begin{tabular}{r|ccccc}
Degree $k$                  &  2      & 3      & 4      & 5    & 6  \\ \hline
Sum-of-rational \eqref{eq:peak_rat_sos}    &  0.063 & 0.922 &0.766 &3.093 & 5.438
\\
Lifted \eqref{eq:peak_rat_sos_lift} & 0.063 & 0.157 & 1.094 & 13.266 &
32.312
 \\
Cleared    \eqref{eq:peak_rat_sos_clear} & 0.110 & 0.281 & 0.531& 1.015 &
1.9690

    \end{tabular}    
    \label{tab:mm_time}
\end{table}

Degree $k=6$ bounds (black dotted lines) and sample trajectories (colored curves) starting from $X_0$ are plotted in Figure \ref{fig:mm}. The $k=6$ bounds from Table \ref{tab:mm} are visually indistinguishable on Figure \ref{fig:mm}. The  red dot marks the location of the stable equilibruim point $x_{eq}$.

\begin{figure}[h]
    \centering
    \includegraphics[width=\linewidth]{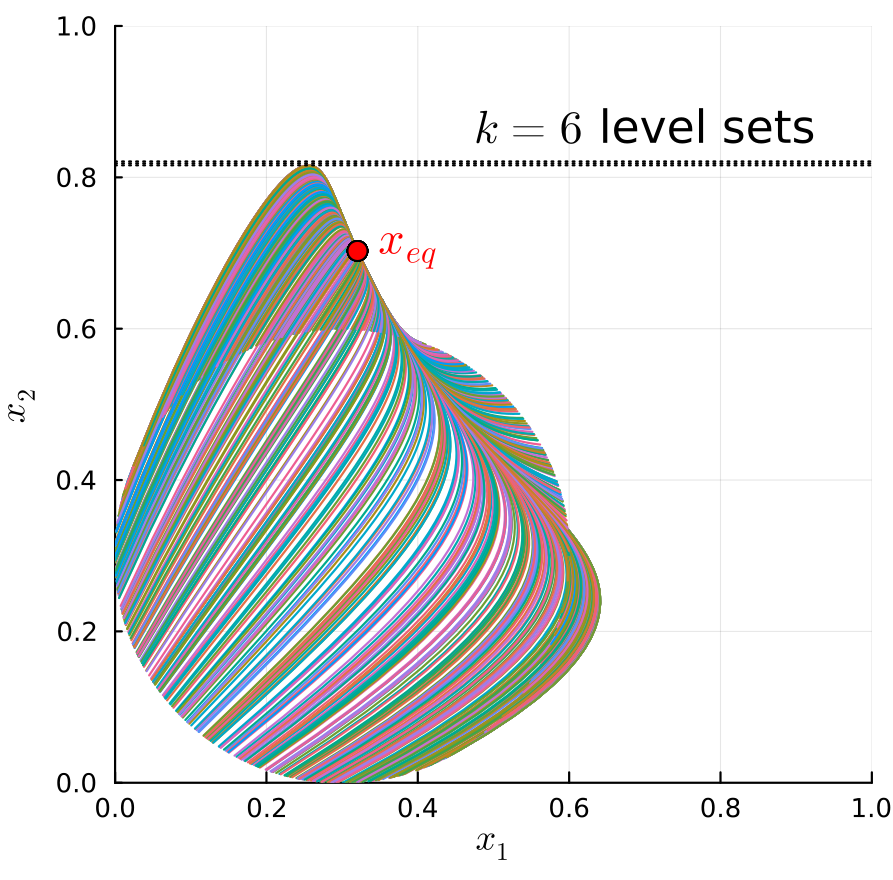}
    \caption{Trajectories and $k=6$ bounds for \eqref{eq:ex_mm}, along with position of the unique equilibrium point $x_{eq}$}
    \label{fig:mm}
\end{figure}

\subsection{Three-state Rational Twist System}

The second example involves a rational modification of the Twist system from Equation (37) of \cite{miller2023distance}. The three-state rational twist system is expressed in terms of matrix parameters $(A, B)$ 
    \begin{align}
    \label{eq:twist_parameters}
    A &= \begin{bmatrix}-1 & 1 & 1\\ -1 &0 &-1\\ 0 & 1 &-2\end{bmatrix} &  B &=  \frac{1}{2}\begin{bmatrix}-1 & 0 & -1\\ 0 &1 &1\\ 1 & 1 &0\end{bmatrix},
       \end{align}
to form the expression of ($\forall i \in 1..3$):
\begin{align}
    \label{eq:twist_rat_dynamics}
        \dot{x}_i(t) &= \sum_{j=1}^3 3 B_{ij} x_j + \frac{A_{ij}x_j - 4 B_{ij}x_j^3}{0.5+x_i^2}.
\end{align}
The polynomial Twist dynamics in \cite{miller2023distance} lacks the $(0.5+x_i^2)$ denominators. 
Peak estimation of \eqref{eq:twist_rat_dynamics} occurs over the state space of $X = [-1, 1]^3$ in a time horizon of $T=6$. The initial set $X_0$ is the two-dimensional box $X \mid_{x_3 = 0}$. It is desired to upper-bound the peak value of $p(x) = x_3^2$ along trajectories of \eqref{eq:twist_rat_dynamics}. Tables \ref{tab:twist} and \ref{tab:twist_time} compile computed upper-bounds of $p(x)$ along rational Twist trajectories and solution timing, in the same style as in Tables \ref{tab:mm} and \ref{tab:mm_time}. The Sum-of-Ratios program \eqref{eq:peak_rat_sos} returns a bound of $P^* \leq 1$ at degree $k=1$.

\begin{table}[h]
    \centering
    \caption{Bounds for Rational Twist \eqref{eq:twist_rat_dynamics}}
    \begin{tabular}{r|ccccc}
Degree $k$                  &  2      & 3      & 4      & 5    & 6  \\ \hline
Sum-of-rational \eqref{eq:peak_rat_sos}       & 0.9652 & 0.4321 & 0.3590 & 0.3501 & 0.3498\\
Lifted \eqref{eq:peak_rat_sos_lift} & 1 & 1 & 1 & 1 & 1  \\
Cleared    \eqref{eq:peak_rat_sos_clear} & 1 & 1 & 1 & 1 & 1
    \end{tabular}    
    \label{tab:twist}
\end{table}

\begin{table}[h]
    \centering
    \caption{Timing (seconds) for Table \ref{tab:twist}}
    \begin{tabular}{r|ccccc}
Degree $k$                  &  2      & 3      & 4      & 5    & 6  \\ \hline
Sum-of-rational \eqref{eq:peak_rat_sos}    &  0.125 & 
 0.562 & 3.125 & 8.390 & 38.280

\\
Lifted \eqref{eq:peak_rat_sos_lift} & 0.344 & 2.125 & 14.250 & 80.703 & 470.313

 \\
Cleared    \eqref{eq:peak_rat_sos_clear} & 
0.219 &  0.891 &  2.672 & 7.156 &25.156
    \end{tabular}    
    \label{tab:twist_time}
\end{table}

A lower bound on $P^*$ acquired through sampling (gridding the plane $X_0$) is $P^* > 0.3489$. Sampled trajectories and the $p(x) = x_3^2$ level set at the $k=6$ bound (Problem \ref{prob:peak_rat_sos}) for the rational Twist system are plotted in Figure \ref{fig:twist}.

\begin{figure}[h]
    \centering
    \includegraphics[width=\linewidth]{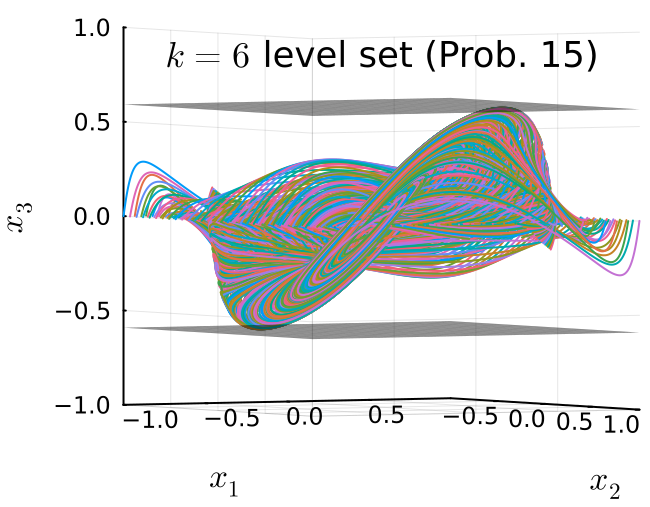}
    \caption{Trajectories and $k=6$ bound for \eqref{eq:twist_rat_dynamics}}
    \label{fig:twist}
\end{figure}


\section{Conclusion}

\label{sec:conclusion}

This paper presents a scheme to perform peak estimation of rational systems. The sum-of-rational optimization technique from \cite{bugarin2016minimizing} is used to reduce the complexity of resulting moment-\ac{SOS}-derived \acp{SDP}. This decomposition scheme is nonconservative if all denominator polynomials $D_\ell$ are positive and a compact set is considered (generating assumptions A1-A6 in the dynamical systems setting). Effectiveness of our technique was demonstrated on example systems.

Future work includes investigating methods to reduce conservatism of assumptions A1-A6 (such as if $D_\ell > 0$ only along the graph of trajectories starting at $X_0$). Other options include applying our method towards rigid body kinematics,  decomposing network structure in dynamics \cite{schlosser2020sparse} through sparse sum-of-rational optimization (Theorem 3.2 of \cite{bugarin2016minimizing}), and analyzing non-\ac{ODE} rational system models.



\section{Acknowledgements}

The authors would like to thank Mohamed Abdalmoaty, Mario Sznaier, Didier Henrion, and Matteo Tacchi for their advice and support.


\appendices
\renewcommand{\thesubsection}{\Alph{subsection}}
\section{Proof of Strong Duality}
\label{app:duality}

This appendix will prove Theorem \ref{thm:strong_dual}'s statement of strong duality between 
\eqref{eq:peak_rat_meas} and \eqref{eq:peak_rat_cont}.
Arguments and notation conventions from Theorem 2.6 of \cite{tacchi2021moment} will be used in this proof.

\subsection{Weak Duality}

We will express the variables of \eqref{eq:peak_rat_meas} and \eqref{eq:peak_rat_cont} as: 
\begin{subequations}
    \label{eq:variables}
\begin{align}
    \bbmu &= (\mu_0, \mu_p, \mu, \{\nu_\ell\}) \\   
    \bell &= (\gamma, v(t, x), \{q_\ell(t, x)\}).
\end{align}
\end{subequations}

The variable $\bbmu$ lies within the space of 
\begin{subequations}
\begin{align}
    \mathcal{X} &= \mathcal{M}(X_0) \times \mathcal{M}([0, T] \times X)^{L+1} \\
    \mathcal{X}' &= C(X_0) \times C([0, T] \times X)^{L+1},
\intertext{of which the nonnegative subcones of the spaces are}
\mathcal{X}_+ &= \mathcal{M}_+(X_0) \times \mathcal{M}_+([0, T] \times X)^{L+1} \\
    \mathcal{X}'_+ &= C_+(X_0) \times C_+([0, T] \times X)^{L+1}.
\end{align}
\end{subequations}

The variable $\bell$ is contained within the spaces of
\begin{subequations}
    \begin{align}
    \mathcal{Y} &= 0 \times C^1([0, T] \times X)' \times \mathcal{M}([0, T] \times X)^L\\
        \mathcal{Y}' &= \R \times C^1([0, T] \times X) \times C([0, T] \times X)^L.
    \end{align}
\end{subequations}

In accordance with the convention of Theorem 2.6 of \cite{tacchi2021moment}, we will also use the symbols $\mathcal{Y}_+ = \mathcal{Y}$ and $\mathcal{Y}'_+ = \mathcal{Y}'$.

The variables $\bbmu$ and $\bell$ satisfy
\begin{align}
    \bbmu &\in \mathcal{X}_+ & & \bell \in \mathcal{Y}'.
\end{align}

The cost and constraint vectors of \eqref{eq:peak_rat_meas} are
\begin{align}
    \mathbf{c} &= [0, p(x), 0, \0] \\
    \mathbf{b} &= [1, 0, \0].
\end{align}

The objectives of \eqref{eq:peak_rat_meas} and \eqref{eq:peak_rat_cont} respectively are
\begin{align}
    \inp{\mathbf{c}}{\bbmu} &= \inp{p(x)}{\mu_p} & \inp{\bell}{\mathbf{b}} &= \gamma.
\end{align}

An affine map $\mathcal{A}$ with adjoint $\mathcal{A}^\dagger$ can be defined with
\begin{subequations}
    \begin{align}
        \A(\bbmu) &= \begin{bmatrix}      
        \mu_p - \delta_0 \otimes\mu_0 - \Lie_{f_0}^\dagger \mu - \textstyle \sum_{\ell=1}^\ell (N_\ell \cdot \nabla_x)^\dagger \nu_\ell \\ \mu - D^\dagger_\ell \nu_\ell \\ \inp{1}{\mu_0}\end{bmatrix} \\
        \A'(\bell) &= \begin{bmatrix}
            \gamma - v(0,x) \\
            v(t, x) - p(x) \\
            \Lie_{f_0} v(t, x) + \textstyle \sum_{\ell=1}^L q_\ell(t, x) \\
            D_\ell(t, x) q_\ell(t, x) - N_\ell(t, x) \cdot \nabla_x v(t, x)
        \end{bmatrix}.
    \end{align}
\end{subequations}

Problems \ref{eq:peak_rat_meas} and \ref{eq:peak_rat_cont} may be placed into standard conic form with weak duality \cite{barvinok2002convex}:
\begin{align}
    p^*_r =& \sup_{\boldsymbol{\mu} \in \mathcal{X}_+} \inp{\mathbf{c}}{\boldsymbol{\mu}} & & \mathbf{b} - \A(\boldsymbol{\mu}) \in \mathcal{Y}_+. \label{eq:peak_rat_meas_std}\\
\intertext{The corresponding weak dual of \eqref{eq:peak_rat_meas_std} is \cite{barvinok2002convex}}
    d^*_r = &\inf_{\boldsymbol{\ell} \in \mathcal{Y}'_+} \inp{\boldsymbol{\ell}}{\mathbf{b}}
    & &\A'(\boldsymbol{\ell}) - \mathbf{c} \in \mathcal{X}_+. \label{eq:peak_rat_con_std}
\end{align}

\subsection{Strong Duality}

Strong duality under A1-A5 is attained if the following three conditions are fulfilled (sufficient condition):
\begin{itemize}
    \item[R1] All measures in $\bbmu$ that are feasible for \eqref{eq:peak_rat_meas} are bounded.
    \item[R2] There exists a feasible measure variable $\bbmu$ solving \eqref{eq:peak_rat_meas}.    
    \item[R3] The elements of $(\A, \mathbf{b}, \mathbf{c})$ are continuous in $(t, x)$.
\end{itemize}

Requirement R1 will be proven by noting that all measures in $\bbmu$ are supported on compact sets (A1), and that all feasible measures in $\bbmu$ have finite mass. 
Requirement R2 is satisfied by choosing a feasible trajectory parameterized by $(t^*, x_0^*)$ and using the construction process from the proof of Theorem \ref{eq:no_relaxation_upper_bound} to create measures $(\mu_0, \mu_p, \mu)$. The measures $\{\nu_\ell\}$ may be uniquely defined from $\mu$ by constraint \eqref{eq:peak_rat_meas_abscont} as:
\begin{align}
    & \forall\omega \in C([0, T] \times X): \nonumber\\
    & \ \quad \inp{\omega}{\nu_\ell} = \int_{0}^{t^*} \frac{\omega(t, x(t \mid x_0^*))}{ D_\ell (t, x(t \mid x_0^*))} dt.
\end{align}

Lastly, we examine requirement R3. The vector $\mathbf{c}$ is continuous in $(t, x)$, because $p$ is continuous by A3. The vector $\mathbf{b}$ is continuous in $(t, x)$ because it is a constant. The polynomial structure of $(f_0, N_\ell, D_\ell)$ ensures that $\forall \bell \in \mathcal{Y}_+:$ the vector $\A^\dagger(\bell)$ is continuous in $(t, x)$ (because the product of continuous functions is continuous).

Strong duality of $p^*_r = d^*_r$ is therefore proven.

\subsection{Duality without polynomial structure}

The weak dual of \eqref{eq:peak_meas} is \cite{cho2002linear}:
\begin{subequations}
\label{eq:peak_cont}
    \begin{align}
      d^* = &\ \inf_{\gamma \in \R} \gamma \\      
      & \gamma \geq v(0, x) & & \forall x \in X_0\\      
      & v(t, x) \geq p(x)  & & \forall (t, x) \in [0, T] \times X \\
      & \Lie_{f} v(t, x) \leq 0 \quad & & \forall (t, x) \in [0, T] \times X \label{eq:peak_cont_lie}\\      
      & v \in C^1([0, T] \times X).
    \end{align}
\end{subequations}
If $(f_0, N_\ell, D_\ell)$ are allowed to be discontinuous, then the nonsmooth nature of $f$ in \eqref{eq:rational_dynamics} violates the assumption of $(t, x)$-continuity in requirement R3. As a result, program \eqref{eq:peak_cont} would be a weak dual to Problem \ref{prob:peak_meas}, but is not necessarily guaranteed to be a strong dual.
\section{Other Rational-Peak SOS Methods}

\label{app:other_methods}

This appendix presents \ac{SOS} programs for peak estimation of rational functions based on clearing to common denominators \cite{parker2021study} and by introducing new lifting variables and equality constraints \cite{magron2018semidefinite, newton2023rational}.

\subsection{Clearing to Common Denominators}

Assumptions A1-A7 will be in effect for this subsection.

Let us define the polynomials $\Phi(t, x) = \prod_{\ell=1}^L D_\ell(t, x)$ and $\Phi_\ell(t, x) = \prod_{\ell\neq \ell'} D_{\ell'}(t, x)$.

The method of clearing to common denominators replaces the term
\begin{align}
    \Lie_f v(t, x) \leq 0 & & \forall (t, x) \in [0, T] \times X \\
    \intertext{from \eqref{eq:peak_cont_lie} with the equivalent (under A5) expression of}
    \Phi(t, x) \Lie_f v(t, x) \leq 0 & & \forall (t, x) \in [0, T] \times X \label{eq:clear}.
\end{align}

The resultant degree-$k$ \ac{SOS} truncation of \eqref{eq:peak_cont} for peak estimation under the clearing term \eqref{eq:clear} is
\begin{subequations}
\label{eq:peak_rat_sos_clear}
    \begin{align}
      d^*_{c, k} = &\ \inf_{\gamma \in \R, v} \gamma \\      
      & \gamma - v(0, x) \in \Sigma[X_0]_{\leq 2k}\\
      & \quad v(t, x) - p(x) \in \Sigma[([0, T] \times X)]_{\leq 2k}\\
      &-\Phi(t, x) \Lie_{f_0} v(t, x) \textstyle -\sum_{\ell=1}^L \Phi_\ell(t, x) (N_\ell \cdot \nabla_x v(t, x)) \label{eq:peak_rat_sos_clear_big} \\
      & \ \qquad \in \Sigma[([0, T] \times X)]_{\leq 2k + 2 \max_\ell \floor{(\deg \Phi N_\ell-1) /2}} \nonumber        \\
      & v \in \R[t, x]_{\leq 2k}.      
    \end{align}
\end{subequations}

The degree of the constraint \eqref{eq:peak_rat_sos_clear_big} will rapidly rise as the number of terms $L$ increases, due to the complexity of forming $\Phi(t, x)$, yielding a maximal-size PSD constraint of size $\binom{1+n+k+\floor{(\deg \Phi-1)/2}}{1+n}$.

\subsection{Lifting Variables}

This subsection will require assumptions A1-A4 and A6 (but not necessarily A5). The work of \cite{magron2018semidefinite} associates each denominator $D_\ell(t, x)$ to a new variable $y_\ell$ in the fashion of an index-1 differential-algebraic equation   \cite{rabier2002theoretical}. The support set of the lifted variable is:
\begin{align}
    \Omega &= \{(t, x, y) \in [0, T] \times X \times \R^L \mid \label{eq:lifting}\\ &\qquad \qquad  \forall \ell=1\ldots L: \ y_\ell D_\ell(t, x) = 1\}\nonumber.
\end{align}
The set $\Omega$ will be compact under assumptions A2 and A5.
The Lie derivative expression for the lifted dynamics is
\begin{align}
    \textstyle \Lie_f v(t, x) = \Lie_{f_0} v(t, x) + \sum_{\ell=1}^L y_\ell (N_\ell \cdot  \nabla_x v(t, x)).
\end{align}
Define $Y$ as the following degree:
\begin{align}
    Y = \max(\deg f_0 - 1, \max_\ell \deg N_\ell).
\end{align}

The lifted-form degree-$k$ \ac{SOS} program for rational peak estimation is
\begin{subequations}
\label{eq:peak_rat_sos_lift}
    \begin{align}
      d^*_{y, k} = &\ \inf_{\gamma \in \R, v} \gamma \\      
      & \gamma - v(0, x) \in \Sigma[X_0]_{\leq 2k}\\
      & \quad v(t, x) - p(x) \in \Sigma[([0, T] \times X)]_{\leq 2k}\\
      &-\Lie_{f_0} v(t, x) \textstyle -\sum_{\ell=1}^L y_\ell (N_\ell \cdot \nabla_x v(t, x)) \label{eq:peak_rat_sos_lift_lie} \\
      & \ \qquad \in \Sigma[\Omega]_{\leq 2k + 2 \floor{Y/2}} \nonumber        \\
      & v \in \R[t, x]_{\leq 2k}.      
    \end{align}
\end{subequations}

The Lie derivative constraint in \eqref{eq:peak_rat_sos_lift_lie} involves $1+n+L$ variables. This results in a maximal size PSD constraint of $\binom{1+n+L+k+\floor{Y/2}}{k+\floor{Y/2}}$, which can be optionally reduced to $\binom{2+n+k+\floor{Y/2}}{k+\floor{Y/2}}$ under quotient and correlative sparsity methods at the cost of increased finite-degree conservatism.




\bibliographystyle{ieeetr}
\bibliography{references.bib}

\end{document}